\documentclass[preprint,11pt]{elsarticle}

\usepackage{geometry, hyperref, enumerate, amsthm, amsmath, amssymb, amscd, color, mathrsfs,accents}
\usepackage[all]{xy}

\makeatletter
\def\ps@pprintTitle{%
     \let\@oddhead\@empty
     \let\@evenhead\@empty
     \def\@oddfoot{\reset@font {\footnotesize\itshape to appear in J. Algebra}\hfil\thepage\hfil%
     \llap{\footnotesize\itshape\today}}% date!
     \let\@evenfoot\@oddfoot}
\makeatother 

\numberwithin{equation}{section}

\newtheorem{Prop}[equation]{Proposition}
\newtheorem{Thm}[equation]{Theorem}
\newtheorem{Lemma}[equation]{Lemma}
\newtheorem{Cor}[equation]{Corollary}
\theoremstyle{definition}
\newtheorem{Def}[equation]{Definition}
\newtheorem{Ex}[equation]{Example}
\newtheorem{Rem}[equation]{Remark}

\newcommand{\Int}{{\rm Int}}

\newcommand{\N}{\mathbb{N}}
\newcommand{\Q}{\mathbb{Q}}
\newcommand{\Z}{\mathbb{Z}}
\newcommand{\R}{\mathbb{R}}

\newcommand{\mcW}{\mathcal{W}}

\newcommand{\Br}{{\rm Br}}

\journal{Journal of Algebra}

\begin{document}

\begin{frontmatter}

\title{Pr\"ufer intersection of valuation domains of a field of rational functions}

\author{Giulio Peruginelli} 
\ead{gperugin@math.unipd.it}
\address{Dipartimento di Matematica "Tullio Levi-Civita",\\ Universit\`a di Padova, Via Trieste 63, 35121 Padova, Italy}

%% use the tnoteref command within \title for footnotes;
%% use the tnotetext command for theassociated footnote;
%% use the fnref command within \author or \address for footnotes;
%% use the fntext command for theassociated footnote;
%% use the corref command within \author for corresponding author footnotes;
%% use the cortext command for theassociated footnote;
%% use the ead command for the email address,
%% and the form \ead[url] for the home page:
%% \title{Title\tnoteref{label1}}
%% \tnotetext[label1]{}
%% \author{Name\corref{cor1}\fnref{label2}}
%% \ead{email address}
%% \ead[url]{home page}
%% \fntext[label2]{}
%% \cortext[cor1]{}
%% \address{Address\fnref{label3}}
%% \fntext[label3]{}

%% use optional labels to link authors explicitly to addresses:
%% \author[label1,label2]{}
%% \address[label1]{}
%% \address[label2]{}

\begin{abstract}
\noindent Let $V$ be a rank one valuation domain with quotient field $K$. We characterize the subsets $S$ of $V$ for which the ring of integer-valued polynomials $\Int(S,V)=\{f\in K[X] \mid f(S)\subseteq V\}$ is a Pr\"ufer domain. The characterization is obtained by means of the notion of pseudo-monotone sequence and pseudo-limit in the sense of Chabert, which generalize the classical notions of pseudo-convergent sequence and pseudo-limit by Ostrowski and Kaplansky, respectively. We show that $\Int(S,V)$ is Pr\"ufer if and only if no element of the algebraic closure $\overline{K}$ of $K$ is a pseudo-limit of a pseudo-monotone sequence contained in $S$, with respect to some extension of $V$ to $\overline{K}$. This result expands a recent result by Loper and Werner.
\end{abstract}

\begin{keyword}
Pr\"ufer domain \sep pseudo-convergent sequence \sep pseudo-limit \sep residually transcendental extension \sep integer-valued polynomial

\MSC Primary 13F05 Secondary  13F20 \sep 13A18

\end{keyword}

\end{frontmatter}

\section{Introduction}

An integral domain $D$ is Pr\"ufer if $D_M$ is a valuation domain for each maximal ideal $M$ of $D$. A Pr\"ufer domain $D$ enjoys an abundance of properties  (see for example \cite{Gilmer}), among which there is the fact that $D$ is integrally closed. By a celebrated result of Krull, every integrally closed domain with quotient field $K$ can be  represented as an intersection of valuation domains of $K$. Conversely, it is of extreme importance  to establish when a given family of valuation domains of a  given field $K$ intersects in a Pr\"ufer domain with quotient field $K$. This problem has also connections to real algebraic geometry, since the real holomorphy rings of a formally real function field is well-known to be a Pr\"ufer domain (see for example \cite[\S 2.1]{FHP}). Different authors have investigated this problem: for example, Gilmer  and Roquette gave explicit construction of Pr\"ufer domains constructed as intersection of valuation domains, or, which is the same thing, as the integral closure of some subring (see \cite{GilmPrufer} and \cite{Roq}, respectively). Recently, Olberding gave a geometric criterion on a subset $Z$ of the Zariski-Riemann space of all the valuation domains of a field in order for the holomorphy ring $\bigcap_{V\in Z}V$ to be a Pr\"ufer domain; this criterion is given in terms of projective morphisms of $Z$, considered as a locally ringed space, into the projective line (see \cite{Olb1}). In \cite{Olb2} Olberding gave a sufficient condition on a family of rank one valuation domains which satisfies certain assumptions so that the intersection of the elements of the family is a Pr\"ufer domain. 

In this paper we focus our attention to the relevant class of polynomial rings called integer-valued polynomials. Classically, given an integral domain $D$ with quotient field $K$ and a subset $S$ of $D$, the ring of integer-valued polynomials over $S$ is defined as:
$$\Int(S,D)=\{f\in K[X] \mid f(S)\subseteq D\}.$$
For $S=D$, we set $\Int(D,D)=\Int(D)$. We refer to \cite{CaCh} for a detailed treatment of this kind of rings. If $D$ is Noetherian, Chabert and McQuillan independently gave sufficient and necessary conditions on $D$ so that $\Int(D)$ is Pr\"ufer (see \cite[Theorem VI.1.7]{CaCh}). Later on, Loper generalized their result to a general domain $D$ (see \cite{LopClass}). The problem of establishing when $\Int(S,D)$ is a Pr\"ufer domain for a general subset $S$ of $D$ is considerably more difficult, see \cite{LS} for a recent survey on this problem. Since a necessary condition for $\Int(S,D)$ to be Pr\"ufer is that $D$ is Pr\"ufer (see for example \cite{LS}), it is reasonable to work locally. Henceforth, we consider  $D$ to be equal to a valuation domain $V$. 

The ring $\Int(S,V)$ can be represented in the following way as an intersection of a family of valuation domains of the field of rational functions $K(X)$ and the polynomial ring $K[X]$ (which likewise can be represented as an intersection of valuation domains lying over the trivial valuation domain $K$):
$$\Int(S,V)=K[X]\cap \bigcap_{s\in S}W_s$$
where, for each $s\in S$, $W_s$ is the valuation domain of those rational functions which are integer-valued at $s$, i.e.: $W_s=\{\varphi\in K(X) \mid \varphi(s)\in V\}$. In the language of Roquette \cite{Roq}, a rational function $\varphi\in K(X)$ is holomorphic at $W_s$ (or, equivalently, $\varphi$ has no pole at $W_s$) if and only if $\varphi$ is integer-valued at $s$.  Clearly, $W_s$ lies over $V$, and, in the case $V$ has rank one, $W_s$ has rank two. The topology on the subspace of the Riemann-Zariski space of $K(X)$ formed by the valuation domains $W_s$, $s\in S$, has been extensively studied in \cite{PerTransc}, when $V$ has rank one: in particular, $\{W_s \mid s\in S\}$ as a subspace of the  Zariski-Riemann space of all the valuation domains of $K(X)$ is homeomorphic to $S$, considered as a subset of $V$, endowed with the $V$-adic topology. 

For a general valuation domain $V$, we have the following well-known result (which is now a special case of the aforementioned result of Loper in \cite{LopClass}):

\begin{Thm}\cite[Lemma VI.1.4, Proposition VI.1.5]{CaCh}\label{IntVPrufer}
Let $V$ be a valuation domain. Then $\Int(V)$ is a Pr\"ufer domain if and only if $V$ is a DVR with finite residue field.
\end{Thm}

The first result about when $\Int(S,V)$ is Pr\"ufer dates back to McQuillan: he showed that if $S$ is a finite set then $\Int(S,V)$ is Pr\"ufer  (more generally, he showed that for a finite subset $S$ of an integral domain $D$, $\Int(S,D)$ is Pr\"ufer if and only if $D$ is Pr\"ufer, see \cite{McQ}). Later on, Cahen, Chabert and Loper turned their attention to infinite subsets $S$ of a valuation domain $V$, and gave the following sufficient condition (here, precompact means that the topological closure of $S$ in the completion of $V$ is compact).

\begin{Thm}\cite[Theorem 4.1]{CCL}\label{ThmCCL}
Let $V$ be a valuation domain and $S$ a subset of $V$. If $S$ is a precompact subset of $V$ then $\Int(S,V)$ is a Pr\"ufer domain.
\end{Thm}

Whether the precompact  condition on $S$ is also a necessary condition or not was a natural question posed in \cite{CCL}. If $V$ is a rank one discrete valuation domain, then it is sufficient and necessary that $S$ is precompact in order for $\Int(S,V)$ to be Pr\"ufer (\cite[Corollary 4.3]{CCL}). Similarly, Park proved recently that if $S$ is an additive subgroup of any valuation domain $V$, then $\Int(S,V)$ is a Pr\"ufer domain if and only if $S$ is precompact (\cite[Theorem 2.7]{MHP}). Unfortunately, already for a non-discrete rank one valuation domain $V$ the precompact condition turned out to be not necessary, as Loper and Werner showed by considering subsets $S$ of $V$ whose elements comprise a pseudo-convergent sequence in the sense of Ostrowski (for all the definitions related to this notion see \S \ref{Gps} below). It is worth recalling that  the first time this notion has been used in the realm of integer-valued polynomials is in two articles of Chabert (see \cite{ChabPolCloVal,ChabIntValValField}). Loper and Werner made a thorough study of the   rings of polynomials which are integer-valued over a pseudo-convergent sequence $E=\{s_n\}_{n\in\N}$ of a rank one valuation domain $V$, obtaining the following characterization of when $\Int(E,V)$ is Pr\"ufer.

\begin{Thm}\cite[Theorem 5.2]{LW}\label{ThmLW}
Let $V$ be a rank one valuation domain and $E=\{s_n\}_{n\in\N}$ a pseudo-convergent sequence in $V$. Then $\Int(E,V)$ is a Pr\"ufer domain if and only if either $E$ is of transcendental type or the breadth ideal of $E$ is the zero ideal.
\end{Thm}

In particular, if $E$ is a pseudo-convergent sequence with non-zero breadth ideal and of transcendental type, then $E$ is not precompact and $\Int(E,V)$ is a Pr\"ufer domain  (\cite[Example 5.12]{LW}). 

In this paper,  we give a sufficient and necessary condition on a general subset $S$ of a rank one valuation domain $V$ so that $\Int(S,V)$ is Pr\"ufer, generalizing the above result by Loper and Werner. Throughout the paper, we assume that $V$ is a rank one valuation domain with maximal ideal $M$ and quotient field $K$. We denote by $v$ the associated valuation and by $\Gamma_v$ the value group. In particular, $\Gamma_v$ is an ordered subgroup of the reals, so that $\Gamma_v\subseteq\R$. Our approach proceeds as follows. We employ a criterion for an integrally closed domain $D$ to be Pr\"ufer (which can be found for example in the book of Zariski and Samuel \cite{ZS2}): it is sufficient and necessary that, for each valuation overring  $W$ of $D$ with center a prime ideal $P$ on $D$, the extension of the residue field of $W$ over the quotient field of $D/P$ is not transcendental. In our setting, a  valuation overring $W$ of $\Int(S,V)$ which does not satisfy the previous property is a residually transcendental extension of $V$ (i.e.: $W$ lies over $V$ and the residue field of $W$ is a transcendental extension of the residue field of $V$). These valuation domains of the field of rational functions have been completely described by Alexandru and Popescu.  Putting together these facts, we show that the lack of the Pr\"ufer property for $\Int(S,V)$ occurs precisely when $S$ contains a pseudo-monotone sequence in the sense of Chabert which admits a pseudo-limit in the algebraic closure of $K$ (with respect to a suitable extension of $V$). These notions generalize  the notions of pseudo-convergent sequence and pseudo-limit in the sense of Ostrowski and Kaplansky, respectively.   

Here is a summary of this paper. In \S \ref{Gps} we introduce the notion of pseudo-monotone sequence and pseudo-limit given by Chabert. In \S \ref{pol clos} we recall a result of Chabert about the fact that the  polynomial closure of a subset $S$ of $V$, defined as the largest subset of $V$ over which all the polynomials of $\Int(S,V)$ are integer-valued, is a topological closure. In \S \ref{Resid transc exten} we recall the aforementioned criterion for an integrally closed domain to be Pr\"ufer and an explicit description by Alexandru and Popescu of residually transcendental extensions of a valuation domain, which are crucial for our discussion. Finally, in \S \ref{final result}, we give our main result which classifies the subsets $S$ of a rank one valuation domain $V$ for which $\Int(S,V)$ is Pr\"ufer (see Theorem \ref{final theorem}). This result is accomplished by describing when an element $\alpha\in K$ is a pseudo-limit of a pseudo-monotone sequence contained in $S$: this happens when a closed ball $B(\alpha,\gamma)=\{x\in K \mid v(x-\alpha)\geq\gamma\}$ is contained in the polynomial closure of $S$. From this point of view, the assumption of Theorem \ref{ThmCCL} is equivalent to the fact that $S$ does not contain any pseudo-monotone sequence, which is a sufficent but not necessary condition for $\Int(S,V)$ to be Pr\"ufer, as the above example of Loper and Werner shows.

\section{Preliminaries}

\subsection{Pseudo-monotone sequences}\label{Gps}

We introduce the following notion, which is given by Chabert in \cite{ChabPolCloVal}. It contains the classical definition of pseudo-convergent sequence of a valuation domain by Ostrowski in \cite{Ostr} and exploited by Kaplansky in \cite{Kap} to describe immediate extensions of a valued field.
\begin{Def}
Let $E=\{s_n\}_{n\in\N}$ be a sequence in $K$. We say that $E$ is a \emph{pseudo-monotone  sequence} (with respect to the valuation $v$) if the sequence $\{v(s_{n+1}-s_n)\}_{n\in\N}$ is monotone, that is, one of the following conditions holds:
\begin{itemize}
\item[i)] $v(s_{n+1}-s_n)<v(s_{n+2}-s_{n+1})$, $\forall n\in\N$.
\item[ii)] $v(s_n-s_m)=\gamma\in\Gamma_v$, for all $n\not=m\in\N$. 
\item[iii)] $v(s_{n+1}-s_n)>v(s_{n+2}-s_{n+1})$, $\forall n\in\N$.
\end{itemize}
More precisely, we say that $E$ is \emph{pseudo-convergent}, \emph{pseudo-stationary} or \emph{pseudo-divergent} in each of the three different cases, respectively. Case i) is precisely the original definition given by Ostrowski in \cite[\S 11, p. 368]{Ostr}. Let $\alpha\in K$. We say that $\alpha$ is a \emph{pseudo-limit} of $E$ in each of the three different cases above if:
\begin{itemize}
\item[i)] $v(\alpha-s_n)<v(\alpha-s_{n+1})$, $\forall n\in\N$, or, equivalently, $v(\alpha-s_n)=v(s_{n+1}-s_n)$, $\forall n\in\N$.
\item[ii)] $v(\alpha-s_n)=\gamma$, for all $n\in\N$.
\item[iii)] $v(\alpha-s_n)>v(\alpha-s_{n+1})$, $\forall n\in\N$, or, equivalently, $v(\alpha-s_{n+1})=v(s_{n+1}-s_n)$,  $\forall n\in\N$.
\end{itemize}
We remark that case i) is the definition of pseudo-limit as given by Kaplansky in \cite{Kap}. Given a subset $S$ of $K$ and an element $\alpha$ in $K$, we say that $\alpha$ is a pseudo-limit of $S$ if $\alpha$ is a pseudo-limit of a pseudo-monotone sequence of elements of $S$.

The following limit in $\R\cup\{\infty\}$ is called the \emph{breadth} of a  pseudo-monotone sequence $E$, as given in \cite{ChabPolCloVal}, which generalizes the definition of Ostrowski for pseudo-convergent sequences (\cite[p. 368]{Ostr}):
$$\delta=\lim_{n\to\infty}v(s_{n+1}-s_n).$$
\end{Def}
Note that since $\{v(s_{n+1}-s_n)\}_{n\in\N}$ is either increasing, decreasing or stationary, the above limit is a well-defined real number and $\delta$ may not be in $\Gamma_v$. In the latter case, $V$ is necessarily not discrete. Note that, if $E=\{s_n\}_{n\in\N}\subset V$ and $\alpha$ is a pseudo-limit of $E$, then it is easy to see that the breadth $\delta$ is greater than or equal to $0$ and $\alpha\in V$. We now give some remarks and further definitions for each of the three cases above.

\subsubsection{Pseudo-convergent sequences}\label{pcv}
\begin{Def}
Let $E=\{s_n\}_{n\in\N}$ be a pseudo-convergent sequence in $V$. The following ideal of $V$:
$$\Br(E)=\{b\in V \mid v(b)>v(s_{n+1}-s_n),\forall n\in\N\}$$
is called the \emph{breadth ideal} of $E$.  

We say that $E$ is of \emph{transcendental type} if $v(f(s_n))$ eventually stabilizes for every $f\in K[X]$. If for some $f\in K[X]$ the sequence $v(f(s_n))$ is eventually strictly increasing then we say that $E$ is of \emph{algebraic type}.
\end{Def}

Clearly, the breadth ideal is the zero ideal if and only if $\delta=+\infty$. If $V$ is a discrete rank one valuation domain (DVR), then the breadth ideal is necessarily equal to the zero ideal. In general, this last condition holds exactly when $E$ is a classical Cauchy sequence and then the definition of pseudo-limit boils down to the classical notion of limit (which in this case is unique). Throughout the paper, to avoid confusion, a pseudo-convergent sequence is supposed to have non-zero breadth ideal, and similary, an element $\alpha\in K$ is a pseudo-limit of a sequence $E$ if $E$ is  a pseudo-convergent sequence in this strict sense.  Moreover, in this case if $\alpha\in K$ is a pseudo-limit for $E$, then $\{\alpha\}+\Br(E)$ is the set of all the pseudo-limits for $E$ (\cite[Lemma 3]{Kap}).

The following easy lemma gives a link between the breadth and the breadth ideal for a pseudo-convergent sequence (the inf is considered in $\R$). 
\begin{Lemma}\label{breadth=inf values breadth ideal}
Let $E=\{s_n\}_{n\in\N}\subset V$ be a pseudo-convergent sequence with non-zero breadth ideal. Let
$$\delta'=\inf\{v(b)\mid b\in\Br(E)\}$$
Then $\delta'=\delta$, the breadth of $E$. Moreover, $\delta\in\Gamma_v\Leftrightarrow \Br(E)$ is a principal ideal.
\end{Lemma}
\begin{proof}
Since $v(s_{n+1}-s_n)<v(b)$, for all $b\in\Br(E)$, we have $\delta\leq \delta'$. Suppose that $\delta<\delta'$, then since $\Gamma_v$ is dense in $\R$ there exists $\gamma\in\Gamma_v$ such that $\delta<\gamma<\delta'$. Let $\gamma=v(a)$, for some $a\in K$. If $a\in \Br(E)$ then $\gamma\geq \delta'$ which is not possible. If $a\notin\Br(E)$, then there exists $N\in\N$ such that for all $n\geq N$ we have $v(s_{n+1}-s_n)\geq \gamma$, which also is not possible. Hence $\delta=\delta'$. The last claim is straightforward.
\end{proof}
In particular, the set of all the pseudo-limits of a pseudoconvergent sequence with non-zero breadth ideal and with a pseudo-limit $\alpha\in K$ is equal to the ball $B(\alpha,\delta)=\{x\in K \mid v(x-\alpha)\geq \delta\}$.

\subsubsection{Pseudo-stationary sequences}\label{Remark pseudo-stationary}
Let $E=\{s_n\}_{n\in\N}\subseteq V$ be a pseudo-stationary sequence. Note that, in this case the breadth $\delta$ of $E$ is by definition in $\Gamma_v$, so that $\delta=v(d)$, for some $d\in K$. Moreover, the residue field $V/M$ is infinite. In fact, if $s_n'=\frac{s_n}{d}$, for each $n\in\N$, then $E'=\{s_n'\}_{n\in\N}\subset V\setminus M$ is a pseudo-stationary sequence with breadth $0$, so that there are infinitely many residue classes modulo the maximal ideal $M$. 

Suppose now that $\alpha\in K$ is pseudo-limit of a pseudo-stationary sequence $E$. In \cite[Remark 4.7]{ChabPolCloVal} it is remarked that any element of $\mathring{B}(\alpha,\gamma)=\{\beta\in K \mid v(\alpha-\beta)>\gamma\}$ is a pseudo-limit of $E$. However, if $\beta\in K$ is such that $v(\alpha-\beta)=\gamma$, then $v(s_n-\beta)\geq \gamma$ for every $n\in\N$. Since for all $n\not=m$ we have $\gamma=v(s_n-s_m)=v(s_n-\beta+\beta-s_m)$, for at most one $n'\in\N$ we may have the strict inequality $v(s_{n'}-\beta)>\gamma$. Hence, up to removing one element from $E$, any element of $B(\alpha,\gamma)$ is a pseudo-limit of $E$. In this broader sense,  any element of $E$ itself is a pseudo-limit of $E$.

\subsubsection{Pseudo-divergent sequences}\label{Pds}

If $V$ is discrete, then there are no pseudo-divergent sequences contained in $V$. On the other hand, if $\alpha\in K$ is a pseudo-limit of a pseudo-divergent sequence $E=\{s_n\}_{n\in\N}\subset K$ with breadth $\gamma$, then the set of all the pseudo-limits in $K$ of $E$ is equal to the open ball $\mathring{B}(\alpha,\gamma)=\{x\in K \mid v(x-\alpha)>\delta\}$ (see \cite[Remark 4.7]{ChabPolCloVal}). Note also that any element $s_k\in E$ is definitively a pseudo-limit of $E$, in the sense that, for all $n>k$ we have $v(s_n-s_k)=v(s_n-s_{n-1})>v(s_{n+1}-s_n)=v(s_{n+1}-s_k)$.

\begin{Rem}\label{V not discrete of V/M infinite}
We have seen that if $V$ admits a pseudo-monotone sequence $E=\{s_n\}_{n\in\N}$ with breadth $\gamma\in\R$, then $V$ is either non-discrete or the residue field $V/M$ is infinite. If $E$ is pseudo-stationary, then $V/M$ is necessarily infinite and if $E$ is pseudo-divergent or pseudo-convergent with non-zero breadth ideal then $V$ is necessarily non-discrete. In particular, the only pseudo-monotone sequences in a DVR are the pseudo-stationary sequences.
\end{Rem}

\subsection{Polynomial closure}\label{pol clos}
\begin{Def}
Let $S$ be a subset of $K$. The \emph{polynomial closure} of $S$ is the largest subset of $K$ over which the polynomials of $\Int(S,V)$ are integer-valued, namely:
$$\overline{S}=\{s\in K \mid \forall f\in\Int(S,V),f(s)\in V\}$$
Equivalently, the polynomial closure of $S$ is the largest subset $\overline{S}$ of $K$ such that $\Int(S,V)=\Int(\overline{S},V)$. A subset $S$ of $K$ such that $S=\overline{S}$ is called \emph{polynomially closed}.
\end{Def}
The main result of Chabert in \cite{ChabPolCloVal} is the following theorem, which will be essential in \S \ref{final result} for the proof of our main result.
\begin{Thm}\cite[Theorem 5.3]{ChabPolCloVal}\label{Thm Chabert}
Let $V$ be a valuation domain of rank one. Then the polynomial closure is a topological closure, that is, there exists a topology on $K$ for which the closed sets are exactly the polynomially closed sets. A basis for the closed sets for this topology is given by the finite unions of closed balls $B(a,\gamma)=\{x\in K \mid v(x-a)\geq \gamma\}$, for $a\in K$ and $\gamma\in\Gamma_v$.
\end{Thm}

\begin{Def}
The topology on the valued field $K$ which has the polynomially closed subsets as closed sets is called \emph{polynomial topology}.
\end{Def}

Chabert observes that the polynomial topology is in general weaker than the $v$-adic topology. They coincide if $V$ is discrete and with finite residue field, but the next example (which we will use in the following) shows that they may differ in general.

\begin{Ex}\label{polclosure ball}
Given $\alpha\in K$ and $\gamma\in\R$, in \cite[Proposition 3.2]{ChabPolCloVal} it is proved that the polynomial closure of the open ball $\mathring{B}(\alpha,\gamma)=\{x\in K\mid v(x-\alpha)>\gamma\}$ (which is closed in the $v$-adic topology) is equal to:
$$\overline{\mathring{B}(\alpha,\gamma)}=\left\{
\begin{array}{ll}
B(\alpha,\overline{\gamma}),\;\textnormal{ where }\overline\gamma=\inf\{\lambda\in\Gamma_v \mid \lambda>\gamma\}, &\textnormal{ if either }v\textnormal{ is discrete or }\gamma\notin\Gamma_v\\
B(\alpha,\gamma),& \textnormal{ otherwise}
\end{array}\right.$$
\end{Ex}

\begin{Rem}\label{open sets polyn topology} In particular, given $\alpha\in K$, the subsets of the form:
$$\bigcap_{i=1}^r\{x\in K \mid v(x-s_i)<\gamma_i\}$$
where $s_i\in K$ and $\gamma_i\in \Gamma_v$ are such that $v(\alpha-s_i)<\gamma_i$, for $i=1,\ldots,r$, form a fundamental system of open neighborhoods of $\alpha$ for the polynomial topology.
\end{Rem}

\section{Residually transcendental extensions}\label{Resid transc exten}

The following criterion, which appears for example in \cite[Theorem 10, chapt. VI, \S 5]{ZS2} or \cite[Theorem 19.15]{Gilmer}, establishes when an integrally closed domain $D$ is Pr\"ufer: $D$ must not admit a valuation overring $V$ whose residue field is a transcendental extension of the quotient field of the residue of $D$ modulo the center of $V$ on $D$. Recall that a valuation overring $V$ of an integral domain $D$ is a valuation domain $V$ contained between $D$ and its quotient field $K$. The center of a valuation overring $V$ of $D$ is the intersection of the maximal ideal of $V$ with $D$.

\begin{Thm}\label{criterionZS}
Let $D$ be an integrally closed domain and $P$ a prime ideal of $D$. Then $D_P$ is a valuation domain if and only if there is no valuation overring $V$ of $D$ centered in $P$ such that the residue field of $V$ is transcendental over the quotient field of $D/P$.
\end{Thm}

By means of this Theorem, we are going to show that an integrally closed domain of the form $\Int(S,V)$, $S\subseteq V$, is not Pr\"ufer exactly when it admits a valuation overring lying over $V$ and whose residue fields extension is transcendental.

\begin{Def}
A valuation domain $\mcW$ of the field of rational functions $K(X)$  is a \emph{residually transcendental extension} of $V=\mcW\cap K$  (or simply residually transcendental extension if $V$ is understood) if the residue field of $\mcW$ is a transcendental extension of the residue field of $V$.
\end{Def}

The residually transcendental extensions of $V$ to $K(X)$ have been completely described by Alexander and Popescu (\cite{AP}). In order to describe these valuation domains, we need to introduce the following class of valuations on $K(X)$.

\begin{Def}\label{Valphagamma}
Let $\alpha\in K$ and $\delta$ an element of a value group $\Gamma$ which contains $\Gamma_v$. For $f\in K[X]$ such that $f(X)=a_0+a_1(X-\alpha)+\ldots+a_n(X-\alpha)^n$, we set:
$$v_{\alpha,\delta}(f)=\inf\{v(a_i)+i\delta \mid i=0,\ldots,n\}$$
The function $v_{\alpha,\delta}$ naturally extends to a valuation on $K(X)$ (\cite[Chapt. VI, \S. 10, Lemme 1]{Bourb}). We denote by $V_{\alpha,\delta}$ the valuation domain associated to $v_{\alpha,\delta}$, i.e.: $V_{\alpha,\delta}=\{\varphi\in K(X) \mid v_{\alpha,\delta}(\varphi)\geq0\}$. Clearly, $V_{\alpha,\delta}$ lies over $V$. We let also $M_{\alpha,\delta}=\{\varphi\in K(X) \mid v_{\alpha,\delta}(\varphi)>0\}$ be the maximal ideal of $V_{\alpha,\delta}$.
\end{Def}
\begin{Rem}\label{descriptionValphagamma}
Note that, if $\gamma\in\Gamma_v$, $\gamma\geq0$ and $d\in V$ is any element such that $v(d)=\gamma$, then it is easy to see that:
\begin{align*}
V_{\alpha,\gamma}=V\left[\frac{X-\alpha}{d}\right]_{M\left[\frac{X-\alpha}{d}\right]},\;\;V_{\alpha,\gamma}\cap K[X]=V\left[\frac{X-\alpha}{d}\right],\;\;M_{\alpha,\gamma}\cap K[X]=M\left[\frac{X-\alpha}{d}\right]
\end{align*}
In general, if $\gamma\in\R$, then $V_{\alpha,\gamma}\cap K[X]=\{f(X)=\sum_{k\geq 0}a_k(X-\alpha)^k\in K[X] \mid v(a_k)+k\gamma\geq0,\forall k\}$. As in \cite[\S 4]{ChabIntValValField}, in this case we set $V[(X-\alpha)/\gamma]$ to be $V_{\alpha,\gamma}\cap K[X]$.
\end{Rem}

In \cite[p. 580]{APZ1} the authors say that $v_{\alpha,\delta}$ is residually transcendental if and only if  $\delta$ has finite order over $\Gamma_v$. For the sake of the reader we give a self-contained proof here.

\begin{Lemma}\label{characterization residually transcendental extensions}
Let $\alpha\in K$ and $\delta$ an element of a value group $\Gamma$ which contains $\Gamma_v$.  Then $v_{\alpha,\delta}$ is residually transcendental if and only if $\delta$ has finite order over $\Gamma_v$, i.e., there exists $n\in\N$ such that $n\delta\in \Gamma_v$.
\end{Lemma}
\begin{proof}
Suppose there exists $n\geq 1$ such that $n\delta=\gamma=v(c)\in\Gamma_v$, for some $c\in K$. Clearly, the $v_{\alpha,\delta}$-adic valuation of $f(X)=\frac{(X-\alpha)^n}{c}$ is zero. We claim that over the residue field of $V$ the polynomial $f(X)$ is transcendental. In fact, suppose there exist $a_{d-1},\ldots,a_0\in V$ such that
$$\overline{f}^{d}+\overline{a_{d-1}}\overline{f}^{d-1}+\ldots+\overline{a_{1}}\overline{f}+\overline{a_{0}}=0$$
that is,
$$g=f^d+a_{d-1}f^{d-1}+\ldots+a_1 f+a_0\in M_{\alpha,\delta}$$
However, if we set $a_d=1$, we have:
$$v_{\alpha,\delta}(g)=\inf\{v(a_i)-in\delta+ni\delta \mid i=0,\ldots,d\}=\inf\{v(a_i) \mid i=0,\ldots,d\}=0$$
which is a contradiction.

Conversely, suppose that $n\delta\notin\Gamma_v$, for each $n\geq1$. Let $g\in V_{\alpha,\delta}\setminus M_{\alpha,\delta}$, say $g(X)=\sum_{i\geq0}^d a_i (X-\alpha)^i$; then we have
$$v_{\alpha,\delta}(g)=\inf\{v(a_i)+i\delta \mid i=0,\ldots,d\}=0\Leftrightarrow v(a_0)=0\;\;\&\;\; v(a_i)+i\delta>0, \forall i=1,\ldots,d$$
because of the assumption on $\delta$. Then $g(X)$ is congruent to $g(0)=a_0$ modulo $M_{\alpha,\delta}$ so that over the residue field $g(X)$ is algebraic.
\end{proof}

\begin{Rem}\label{delta in Gammav} 
Suppose that $\delta\in\Gamma_v$; in particular, $v_{\alpha,\delta}$ is residually transcendental. If we consider the following expansion $f(X)=b_0+b_1\frac{X-\alpha}{d}+\ldots+b_n(\frac{X-\alpha}{d})^n$, where $d\in K$ is such that $v(d)=\delta$, then $v_{\alpha,\delta}(f)=\inf\{v(b_i)\mid i=0,\ldots,n\}$. In particular, by \cite[Chapt. VI, \S10, Prop. 2]{Bourb}, $v_{\alpha,\delta}$ is the unique valuation on $K(X)=K(\frac{X-\alpha}{d})$ for which the image of $\frac{X-\alpha}{d}$ in the residue field is transcendental over $V/M$ (note that $\frac{X-\alpha}{d}$ has valuation zero).
\end{Rem}

Let $\overline{K}$ be a fixed algebraic closure of $K$\footnote{$\overline{K}$ is not to be confused with the polynomial closure of $K$, which is $K$ itself. Since both symbols are now equally customary, we decide to change neither of them.} and $\Gamma_{\overline v}=\Gamma_v\otimes_{\Z}\Q$, the divisible hull of $\Gamma_v$. The following theorem characterizes the residually transcendental extensions of $V$ to $K(X)$ (see also \cite[Theorem 3.11]{Kuh} for an alternative and more recent approach). The theorem holds for any valuation domain (i.e., no matter of its dimension). For the sake of the reader we give a sketch of the proof.

\begin{Thm}\label{characterization residually transcendental exts}\cite[Proposition 2 \& Th\'eor\`eme 11]{AP}
Let $\mcW$ be a residually transcendental extension of $V$ to $K(X)$. Then there exist $\alpha\in\overline{K}$, $\gamma\in\Gamma_{\overline{v}}$ and a valuation $\overline{W}$ of $\overline{K}$ lying over $V$ such that $\mcW=\overline W_{\alpha,\gamma}\cap K(X)$.
\end{Thm}
\begin{proof}  Let $\overline{\mcW}$ be an extension of $\mcW$ to $\overline{K}(X)$. It is clear that $\overline{\mcW}$ is residually transcendental over $\overline{\mcW}\cap\overline{K}$. Thus, without loss of generality we may assume that $K$ is algebraically closed. Now, by \cite[Proposition 2]{AP}, $\mcW$ is the valuation domain associated to a valuation $w$ on $K(X)$ which on a polynomial $f\in K[X]$ is defined as:
$$w(f)=\inf_i\{v(a_i)\},\;\;\textnormal{ if } f(X)=\sum_i a_i (aX-b)^i$$
for some $a,b\in K$, $a\not=0$. Now, if we write $f(X)=\sum_i b_i(X-\alpha)^i$ where $b_i=a_i a^i$ and $\alpha=b/a$, we get that $v(a_i)=v(b_i)-iv(a)$, so finally $w=v_{\alpha,\gamma}$, where $\gamma=-v(a)$ (see also Remark \ref{delta in Gammav}).
\end{proof}

\begin{Def}\label{def Valphagamma}
Given $\alpha\in\overline{K}$, $\gamma\in\Gamma_{\overline{v}}$ and a valuation domain $\overline{W}$ of $\overline{K}$ lying over $V$, we denote by $V_{\alpha,\gamma}^{\overline W}$ the valuation domain $\overline{W}_{\alpha,\gamma}\cap K(X)$.  If  $\overline{W}$ is understood we denote $V_{\alpha,\gamma}^{\overline W}$ by $V_{\alpha,\gamma}$.
\end{Def}

\begin{Rem}\label{example}
Let $(\alpha,\gamma)\in\overline{K}\times\Gamma_{\overline{v}}$ be fixed. The valuation domain $V_{\alpha,\gamma}^{\overline W}$ depends on the extension $\overline W$ of $V$ to $\overline{K}$. For example, let $w,w'$ be the $(2-i)$ and $(2+i)$-adic valuations of $\Q(i)$, respectively, which extend the $5$-adic valuation on $\Q$. Then,
$$
w_{i,1}(-X+2)=1,\;\;w'_{i,1}(-X+2)=0.
$$
In particular, $-X+2$ is a unit in $W'_{i,1}$ and is in the maximal ideal of $W_{i,1}$, so the contractions of these valuation domains to $\Q(X)$ cannot be the same. 

Therefore, whenever we write $V_{\alpha,\gamma}$ without any reference to an extension of $V$ to $\overline{K}$, we are implicitly assuming that such an extension has been fixed in advance. Note that there is no ambiguity in writing $V_{\alpha,\gamma}$ whenever $(\alpha,\gamma)\in K\times\Gamma_v$.

Note also that, given a valuation domain $V_{\alpha,\gamma}^{\overline W}$, where $(\alpha,\gamma)\in\overline{K}\times\Gamma_{\overline{v}}$, we may assume that there exists a finite field extension $F$ of $K$ and a valuation domain $W$ of $F$ lying over $V$ such that $\alpha$ is in $F$ and $\gamma$ is in $\Gamma_w$, the value group of $W$, so that $V_{\alpha,\gamma}^{\overline W}=V_{\alpha,\gamma}^{ W}$.
\end{Rem}

\begin{Rem}\label{Valphagamma polynomial linearly ordered}\label{equality Valphagamma}
It is not difficult to prove that the family of rings $V_{\alpha,\gamma}\cap K[X]$, $\alpha\in K,\gamma\in\R$, has a natural ordering, namely: 
$$V_{\alpha_1,\gamma_1}\cap K[X]\subseteq V_{\alpha_2,\gamma_2}\cap K[X]\Leftrightarrow \gamma_1\leq\gamma_2\; \textnormal{ and }v(\alpha_1-\alpha_2)\geq \gamma_1.$$
Equivalently, the above containment holds if and only if $B(\alpha_1,\gamma_1)\supseteq  B(\alpha_2,\gamma_2)$. In particular, 
$$V_{\alpha_1,\gamma_1}\cap K[X]=V_{\alpha_2,\gamma_2}\cap K[X]\Leftrightarrow\gamma_1=\gamma_2\; \textnormal{ and }v(\alpha_1-\alpha_2)\geq \gamma_1,$$
or, equivalently, $B(\alpha_1,\gamma_1)=B(\alpha_2,\gamma_2)$. If this last case holds, then $V_{\alpha_1,\gamma_1}=V_{\alpha_2,\gamma_2}$.

See also \cite[Proposition 1.1]{APZ3}, where the same result is given for any valuation $V$ but only for $\gamma\in \Gamma_v\otimes_{\Z}\Q$.
\end{Rem}

The following lemma is based on a well-known result.
\begin{Lemma}\label{ValphagammacapK[X] not Prufer}
Let $(\alpha,\gamma)\in K\times\Gamma_v$. Then $V_{\alpha,\gamma}\cap K[X]$ is not a Pr\"ufer domain.
\end{Lemma}
\begin{proof}
By Remark \ref{descriptionValphagamma}, $V_{\alpha,\gamma}\cap K[X]=V[\frac{X-\alpha}{d}]$, where $d\in K$ is such that $v(d)=\gamma$. It is a well-known result that $V[\frac{X-\alpha}{d}]$ is not a Pr\"ufer domain. %
\end{proof}

\begin{Lemma}\label{polynomial ring not Prufer}
Let $\overline\mcW$ be a valuation domain of $\overline{K}(X)$ such that $\overline\mcW\cap\overline{K}[X]$ is not Pr\"ufer. Then $\overline\mcW\cap K[X]$ is not Pr\"ufer. 
\end{Lemma}
\begin{proof}
If $R=\overline\mcW\cap K[X]$ is Pr\"ufer, then its integral closure $\overline{R}$ in $\overline{K}(X)$ is Pr\"ufer, because $K(X)\subseteq\overline{K}(X)$ is an algebraic extension. But it is immediate to see that $\overline{R}\subseteq \overline\mcW\cap\overline{K}[X]$: in fact, $R\subset K[X]\Rightarrow \overline{R}\subset \overline{K}[X]$ and $R\subset V=\overline\mcW\cap K(X)$ implies that $\overline{R}$ is contained in the integral closure of $V$ in $\overline K(X)$, which is contained in $\overline\mcW$. In particular, $\overline\mcW\cap\overline{K}[X]$ would be Pr\"ufer, a contradiction.
\end{proof}

The following easy result shows that if $V_{\alpha,\gamma}$, $(\alpha,\gamma)\in K\times\Gamma_v$, is a valuation overring  of $\Int(S,V)$, then $\alpha\in V$ and $\gamma\geq0$.
\begin{Lemma}\label{overrings of V[X]}
Let $(\alpha,\gamma)\in K\times\Gamma_v$. We have $V[X]\subset V_{\alpha,\gamma}$ if and only if $\alpha\in V$ and $\gamma\geq0$.
\end{Lemma}
\begin{proof}
The statement follows immediately from the equality $v_{\alpha,\gamma}(X)=\min\{\gamma,v(\alpha)\}$.
\end{proof}

\begin{Thm}\label{criterion 1 general case}
Let $R\subseteq K[X]$ be an integrally closed domain with quotient field $K(X)$ and such that $D=R\cap K$ is a Pr\"ufer domain with quotient field $K$. Then $R$ is Pr\"ufer if and only if there is no valuation overring $V$ of $D$, an extension $\overline{W}$ of $V$ to $\overline{K}$ and $(\alpha,\gamma)\in \overline K\times\Gamma_{\overline v}$ such that $R\subset V_{\alpha,\gamma}^{\overline{W}}$.
\end{Thm}
The theorem is false without the assumption $R\subseteq K[X]$. For example, $R=V_{\alpha,\gamma}$, $(\alpha,\gamma)\in K\times\Gamma_v$, is a Pr\"ufer domain.
\begin{proof}
Suppose that, for some valuation overring $V$ of $D$ there exist an extension $\overline{W}$ of $V$ to $\overline{K}$ and $(\alpha,\gamma)\in \overline K\times\Gamma_{\overline v}$ such that $R\subset V_{\alpha,\gamma}^{\overline{W}}$. This last condition is equivalent to $K[X]\cap V_{\alpha,\gamma}^{\overline{W}}$. By Lemma \ref{ValphagammacapK[X] not Prufer} and \ref{polynomial ring not Prufer}, $K[X]\cap V_{\alpha,\gamma}^{\overline{W}}$ is not Pr\"ufer, which implies that $R$ is not Pr\"ufer, since an overring of a Pr\"ufer domain is Pr\"ufer.

Conversely, if $R$ is not Pr\"ufer, then by Theorem \ref{criterionZS} there exists a valuation overring $\mcW$ of $R$ with maximal ideal $M_{\mcW}$ such that $R/(M_{\mcW}\cap R)\subset \mcW/M_{\mcW}$ is a transcendental extension. Note that $\mcW\cap K=V$ is a valuation overring of $D$, and since the latter ring is Pr\"ufer, $D_P=V$, where $P$ is the center of $V$ in $D$. We claim that $\mcW$ is a residually transcendental extension of $V$, so that by Theorem \ref{characterization residually transcendental exts} we have $\mcW=V_{\alpha,\gamma}^{\overline{W}}$, for some extension $\overline{W}$ of $V$ to $\overline{K}$ and $(\alpha,\gamma)\in\overline{K}\times\Gamma_{\overline{v}}$. Indeed, we have the following diagram:
$$
\xymatrix{
&\mcW/M_{\mcW}\\
R/(M_{\mcW}\cap R)\ar[ur]&\\
&V/M_V\ar[uu]\\
D/P\ar[ur]\ar[uu]&
}
$$
By Theorem \ref{criterionZS} $D/P\subset V/M_V$ is not a transcendental extension (because $D$ is assumed to be Pr\"ufer), therefore the extension $V/M_V\subset \mcW/M_{\mcW}$ is transcendental and thus $\mcW$ is a residually transcendental extension of $V$. The proof is now complete by Theorem \ref{characterization residually transcendental exts}.
\end{proof}

\section{Pseudo-monotone sequences and polynomial closure}\label{final result}

For the next lemma, see also \cite[Lemma 4.1, Theorem 4.3, Proposition 4.5]{ChabIntValValField}.
\begin{Lemma}\label{polynomial ring and int balls}
Let $\alpha\in K$ and $\gamma\in\R$. We have
\begin{equation}\label{polynomial ring contained in int ring}
V[(X-\alpha)/\gamma]\subseteq \Int(B(\alpha,\gamma),V)
\end{equation}
In particular, if $S\subseteq V$ is such that $\Int(S,V)\subset V_{\alpha,\gamma}$, then $B(\alpha,\gamma)\subseteq \overline{S}$. The other containment of (\ref{polynomial ring contained in int ring}) holds if and only if either $V$ is not discrete or $V/M$ is infinite. In particular, if one of these last conditions holds, then $B(\alpha,\gamma)\subseteq \overline{S}$ implies $\Int(S,V)\subset V_{\alpha,\gamma}$.
\end{Lemma}
\begin{proof}
It is straighforward to verify the containment (\ref{polynomial ring contained in int ring}). Moreover, if $\Int(S,V)\subset V_{\alpha,\gamma}$, we have 
$$\Int(S,V)=\Int(\overline{S},V)\subseteq V_{\alpha,\gamma}\cap K[X]=V[(X-\alpha)/\gamma]$$
(the last equality follows by Remark \ref{descriptionValphagamma}), so the second claim follows. The third claim follows by \cite[Proposition 4.5]{ChabIntValValField}.
Finally, the last claim is straightforward.
\end{proof}

We remark that the last statement is false if $V$ is a DVR with finite residue field: in fact, in that case $\Int(V)$ is Pr\"ufer by Theorem \ref{IntVPrufer} so by Theorem \ref{criterion 1 general case} $\Int(V)\not\subset V_{0,0}$ (note that $V=B(0,0)$).

The following important result by Chabert  shows the connection between pseudo-monotone sequences and polynomial closure.
\begin{Prop}\cite[Prop. 4.8]{ChabPolCloVal}\label{generalized pseudo sequence and polynomial closure}
Let $S\subseteq V$ be a subset. Let $\{s_n\}_{n\in\N}$ be a pseudo-monotone sequence in $S$ with breadth $\gamma\in\R$ and pseudo-limit $\alpha\in V$. Then $B(\alpha,\gamma)\subseteq\overline{S}$, or, equivalently $\Int(S,V)\subset V_{\alpha,\gamma}$. 
\end{Prop}
Note that, under the assumption of Proposition \ref{generalized pseudo sequence and polynomial closure}, by Remark \ref{V not discrete of V/M infinite} either $V$ is not discrete or its residue field is infinite, so the last equivalence of Proposition \ref{generalized pseudo sequence and polynomial closure} follows by Lemma \ref{polynomial ring and int balls}.

The aim of this section is to show that Proposition \ref{generalized pseudo sequence and polynomial closure} can be reversed, in the sense that if $B(\alpha,\gamma)$ is the largest ball centered in $\alpha\in K$ which is contained in the polynomial closure of $S$, then there exists a pseudo-monotone sequence $E$ of $S$ with pseudo-limit $\alpha$ and breadth $\gamma$.

\begin{Rem}\label{breadth not in Gamma_v}
Let $E=\{s_n\}_{n\in\N}\subset V$ be a pseudo-monotone sequence in $V$ with breadth $\gamma$ and pseudo-limit $\alpha\in V$. Suppose that $\gamma=v(d)\in \Gamma_v$, for some $d\in V$. By Proposition \ref{generalized pseudo sequence and polynomial closure} and Remark \ref{descriptionValphagamma}, we have
\begin{equation}\label{IntEVsubseteqValphagammacapK[X]}
\Int(E,V)\subseteq V\left[\frac{X-\alpha}{d}\right]
\end{equation}
If $E$ is either pseudo-stationary or pseudo-divergent, then $v(s_n-\alpha)\geq v(d)$, so the containment in (\ref{IntEVsubseteqValphagammacapK[X]}) is an equality. If $E$ is pseudo-convergent, then $v(s_n-\alpha)<\gamma$ for all $n\in\N$, so we only have the strict containment in (\ref{IntEVsubseteqValphagammacapK[X]}). 

If $E$ is pseudo-convergent or pseudo-divergent, then $\gamma$ may not be in $\Gamma_v$ and may also not be torsion over $\Gamma_v$ (in which case $V_{\alpha,\gamma}$ is not residually transcendental over $V$, by Lemma \ref{characterization residually transcendental extensions}). However, if $\gamma\notin\Gamma_v$, then there exists $\gamma'\in\Gamma_{v}$, $\gamma'>\gamma$ so that $\Int(E,V)\subseteq V_{\alpha,\gamma}\cap K[X]\subset V_{\alpha,\gamma'}\cap K[X]\subset V_{\alpha,\gamma'}$, and $V_{\alpha,\gamma'}$ is residually transcendental over $V$. Therefore, by Theorem \ref{criterion 1 general case}, $\Int(E,V)$ is not a Pr\"ufer domain.
\end{Rem}
Let $\alpha\in K$ and $\gamma\in\Gamma_v$. For the next lemma, we set
$$\partial B(\alpha,\gamma)=\{x\in K \mid v(x-\alpha)=\gamma\}$$
\begin{Lemma}\label{pol closure frontier}
Let $\alpha\in V$ and $\gamma\in \Gamma_v$. If either $V$ is not discrete or $V/M$ is infinite then
$$\overline{\partial B(\alpha,\gamma)}=B(\alpha,\gamma)$$
If $V$ is a DVR with finite residue field, then $\partial B(\alpha,\gamma)$ is polynomially closed.
\end{Lemma}
\begin{proof}
Let $d\in V$ be such that $v(d)=\gamma$. Note that under the isomorphism $X\mapsto\frac{X-\alpha}{d}$, $\partial B(\alpha,\gamma)=\{x\in K \mid v(x-\alpha)=\gamma\}$ is isomorphic to $V^*=V\setminus M$ and similarly $B(\alpha,\gamma)$ is isomorphic to $V$. Therefore, $\overline{\partial B(\alpha,\gamma)}=B(\alpha,\gamma)$ if and only if $\overline{V^*}=V$. We prove now the last equality under the current hypothesis. 

If $V/M$ is infinite, then there exists a sequence $E=\{s_n\}_{n\in\N}\subset V$ and an element $s\in V^*$  such that $v(s_n-s_m)=v(s_n-s)=0$, $\forall n\not=m$. Thus, $E$ is pseudo-stationary with pseudo-limit $s$ and by Proposition \ref{generalized pseudo sequence and polynomial closure} we may conclude.

If $V/M$ is finite, let $V^*=\bigcup_{i=1,\ldots,n}(a_i+M)$, where $a_i\notin M$, $\forall i=1,\ldots,n$. Since the polynomial closure in this context is a topological closure by Theorem \ref{Thm Chabert}, we have $\overline{V^*}=\bigcup_{i=1,\ldots,n}(\overline{a_i+M})=\bigcup_{i=1,\ldots,n}(a_i+\overline{M})$ by \cite[Proposition IV.1.5, p. 75]{CaCh}. The polynomial closure of $M$ is equal either to $V$ if $V$ is not discrete or to $M$ itself if $V$ is discrete (see Example \ref{polclosure ball}). The proof is complete.
\end{proof}

For a subset $S$ of $V$ such that $\Int(S,V)$ is not Pr\"ufer, we know by Theorem \ref{criterion 1 general case} that there exist an extension $\overline{W}$ of $V$ to $\overline{K}$ and $(\alpha,\gamma)\in\overline{K}\times \Gamma_{\overline{v}}$, $\Gamma_{\overline{v}}=\Gamma_v\otimes_{\Z}\Q$, such that $\Int(S,V)\subset V_{\alpha,\gamma}^{\overline{W}}$. The next two propositions show that it is sufficient to consider the case $(\alpha,\gamma)\in V\times\Gamma_v$.

\begin{Prop}\label{integral extension int rings}
Let $S$ be a subset of an integrally closed domain $D$ with quotient field $K$. Let $F$ be an algebraic extension of $K$ and $D_F$ the integral closure of $D$ in $F$. Then the integral closure of $\Int(S,D)$ in $F(X)$ is the ring $\Int(S,D_F)$. 
\end{Prop}
\begin{proof}
It is well-known that $\Int(S,D_F)\subset F[X]$ is integrally closed (see  \cite[Proposition IV.4.1]{CaCh}), so we have just to show that every element of $\Int(S,D_F)$ is integral over $\Int(S,D)$. Up to enlarging the field $F$, we may assume that $F$ is normal over $K$ (e.g., the algebraic closure of $K$). We are going to show that $\Int(S,D)\subset\Int(S,D_F)$ is an integral ring extension under this further assumption. Let then $f\in \Int(S,D_F)$, since $f\in F[X]$, we know that $f$ satisfies a monic equation over the polynomial ring $K[X]$:
$$f^n+g_{n-1}f^{n-1}+\ldots+g_1f+g_0=0,\;\;g_i\in K[X],\;i=0,\ldots,n-1.$$
We claim that $g_i\in\Int(S,D)$, for $i=0,\ldots,n-1$. Let $\Phi(T)=T^n+g_{n-1}T^{n-1}+\ldots+g_0\in K[X][T]$. The roots of $\Phi(T)$ are exactly the conjugates of $f$ under the action of the Galois group ${\rm Gal}(F/K)$, which acts on the coefficients of the polynomial $f$. If $\sigma\in{\rm Gal}(F/K)$, then $\sigma(f)\in F[X]$, and, more precisely, $\sigma(f)\in \Int(S,D_F)$. In fact, for each $s\in S\subset K$, since $\sigma$ leaves each element of $K$ invariant, we have $\sigma(f)(s)=\sigma(f(s))$ which still is an element of $D_F$ (which likewise is left invariant under the action of ${\rm Gal}(F/K)$). Now, since each coefficient $g_i(X)$ of $\Phi(T)$ lies in $K[X]$ and is an elementary symmetric function on the elements $\sigma(f), \sigma\in{\rm Gal}(F/K)$, we have that $g_i(s)\in D_F\cap K=D$, for each $s\in S$, thus $g_i\in\Int(S,D)$, as claimed.
\end{proof}
\begin{Prop}\label{IntSVIntSW}
Let $(\alpha,\gamma)\in F\times\Gamma_{w}$, where $F$ is a finite field extension of $K$ and $W$ is a valuation domain of $F$ lying over $V$. If $S$ is a subset of $V$ such that $\Int(S,V)\subset W_{\alpha,\gamma}\cap K(X)$, then $\Int(S,W)\subset W_{\alpha,\gamma}$.
\end{Prop}
Note that the polynomials in $\Int(S,V)$ have coefficients in $K$, the quotient field of $V$, while the polynomials in $\Int(S,W)$ have coefficients in $F$, the quotient field of $W$.
\begin{proof}
Let $S$ be a subset of $V$ such that $\Int(S,V)\subset W_{\alpha,\gamma}\cap K(X)=V_{\alpha,\gamma}$. The integral closure $V_F$ of $V$ in $F$ is equal to an intersection of finitely many rank 1 valuation domains $W=W_1,\ldots,W_n$ of $F$ lying over $V$. In particular,
$$\Int(S,V_F)=\bigcap_{i=1,\ldots,n}\Int(S,W_i)$$
Let $M_W$ be the maximal ideal of $W$. If $T=V_F\setminus (M_W\cap V_F)$, then $T^{-1}V_F=W$ and since localization commutes with finite intersections we have:
$$T^{-1}\Int(S,V_F)=\bigcap_{i=1,\ldots,n}T^{-1}\Int(S,W_i)$$
Let $f\in K[X]$, say $f(X)=\frac{g(X)}{d}$, for some $g\in V_F[X]$ and $d\in V_F$. Let $\gamma_i=w_i(d)$, for each $i\geq 2$. By the approximation theorem for independent valuations (\cite[Corollaire 1, Chapt. VI, \S 7]{Bourb}), there exists $t\in V_F$ such that $w(t)=0$ and $w_i(t)=\gamma_i$. In particular, $t\in T$. Then $t\cdot f\in \Int(S,W_i)$, so that $T^{-1}\Int(S,W_i)=K[X]$ for all $i\geq 2$. Clearly, $T^{-1}\Int(S,W)=\Int(S,W)$. Hence, 
\begin{equation}\label{localization}
T^{-1}\Int(S,V_F)=\Int(S,W)
\end{equation}
Since $\Int(S,V_F)$ is the integral closure of $\Int(S,V)$ in $F(X)$ by Proposition \ref{integral extension int rings} and $V_{\alpha,\gamma}=W_{\alpha,\gamma}\cap K(X)$ is an overring of $\Int(S,V)$, it follows that $W_{\alpha,\gamma}$ is an overring of $\Int(S,V_F)$. 

Let now $f\in \Int(S,W)$. By (\ref{localization}) there exists $d\in T$ such that $d\cdot f\in \Int(S,V_F)\subset W_{\alpha,\gamma}$. Since $d\in W^*$ is also a unit in $W_{\alpha,\gamma}$, it follows that $f\in W_{\alpha,\gamma}$. This shows that $\Int(S,W)\subset W_{\alpha,\gamma}$, as wanted. Note that, since $S\subseteq V\subset W$, by Lemma \ref{overrings of V[X]}, $\alpha\in W$ and $\gamma\geq0$.
\end{proof}

For a subset $S$ of $V$, $\alpha\in K$ and $\gamma\in\R$, we set 
\begin{align*}
S_{\alpha,<\gamma}=&\{s\in S \mid v(s-\alpha)<\gamma\}\\
S_{\alpha,>\gamma}=&\{s\in S \mid v(s-\alpha)>\gamma\}\\
S_{\alpha,\gamma}=&\{s\in S \mid v(s-\alpha)=\gamma\}
\end{align*}
Note that if $S_{\alpha,\gamma}$ is not empty then $\gamma\in\Gamma_v$.
\begin{Prop}\label{polynomial closure Salpha=gamma}
Let $S\subseteq V$, $\alpha\in V$ and $\gamma\in\Gamma_v$. Then $B(\alpha,\gamma)\subseteq \overline{S_{\alpha,\gamma}}$ if and only if there exists a pseudo-monotone sequence $E=\{s_n\}_{n\in\N}$ in $S_{\alpha,\gamma}$ with breadth $\gamma$ such that $E$ is either pseudo-stationary and with pseudo-limit $\alpha$ or pseudo-divergent and with a  pseudo-limit in $S_{\alpha,\gamma}$. 

In particular, if $V$ is a DVR, then $E$ can only be pseudo-stationary.
\end{Prop}
\begin{proof}

The 'if' part follows from Proposition \ref{generalized pseudo sequence and polynomial closure}.

Conversely, suppose that $B(\alpha,\gamma)\subseteq \overline{S_{\alpha,\gamma}}$. Note that this assumption necessarily implies that either $V$ is non-discrete or $V/M$ is infinite. In fact, $\overline{S_{\alpha,\gamma}}\subseteq\overline{\partial B(\alpha,\gamma)}$ and $\partial B(\alpha,\gamma)$ is polynomially closed if $V$ is a DVR with finite residue field (Lemma \ref{pol closure frontier}), so in this case $\overline{S_{\alpha,\gamma}}$ could not contain $B(\alpha,\gamma)$. 

Suppose that there is no pseudo-divergent sequence $E=\{s_n\}_{n\in\N}$ in $S_{\alpha,\gamma}$ with breadth $\gamma$ and with pseudo-limit $s\in S_{\alpha,\gamma}$. This is equivalent to the following: for each $s\in S_{\alpha,\gamma}$, let $\gamma_s=\inf\{v(s-s')\mid s'\in S_{\alpha,\gamma},v(s'-s)>\gamma\}$. Then $\gamma_s>\gamma$, for each $s\in S_{\alpha,\gamma}$ (note that, a priori, for $s\not=s'\in S_{\alpha,\gamma}$ we have $v(s-s')\geq\gamma$). We construct now a pseudo-stationary sequence in $S_{\alpha,\gamma}$ with pseudo-limit $\alpha$ and breadth $\gamma$. Let $s_1\in S_{\alpha,\gamma}$. Then $\alpha$ belongs to the set $U_1=\{x\in K \mid v(x-s_1)<\gamma_{s_1}\}$, which is open in the polynomial topology by Remark \ref{open sets polyn topology}, and since $\alpha\in\overline{S_{\alpha,\gamma}}$ by assumption, there exists $s_2\in S_{\alpha,\gamma}\cap \{x\in K \mid v(x-s_1)<\gamma_{s_1}\}$. By definition of $\gamma_{s_1}$, we must have $v(s_1-s_2)=\gamma$. Now, we consider the open set $U_2=\{x\in K \mid v(x-s_i)<\gamma_{s_i},i=1,2\}$. Since $\alpha\in U_2$ there exists $s_3\in U_2\cap S_{\alpha,\gamma}$, so that $v(s_3-s_i)<\gamma_{s_i}$, for $i=1,2$, which implies that $v(s_3-s_i)=\gamma$, for $i=1,2$. If we continue in this way, we get a pseudo-stationary sequence $E=\{s_n\}_{n\in\N}\subseteq S_{\alpha,\gamma}$ with  pseudo-limit $\alpha$ and breadth $\gamma$, as wanted.

The last claim follows immediately from \S \ref{Pds}.
\end{proof}

\vskip0.5cm
In the next two results, we consider an integral domain $R\subset K[X]$ with quotient field $K(X)$  such that $R\subset V_{\alpha,\gamma'}$, for some $(\alpha,\gamma')\in V\times\Gamma_v$ (in particular, $R$ is not Pr\"ufer by Theorem \ref{criterion 1 general case}). If $\alpha$ is fixed, then by Remark \ref{Valphagamma polynomial linearly ordered} the set of rings $\{K[X]\cap V_{\alpha,\gamma'} \mid \gamma'\in\Gamma_v\}$ is linearly ordered. In the following we consider the infimum in $\R$ of the set $\{\gamma'\in\Gamma_v \mid R\subset V_{\alpha,\gamma'}\}$.

\begin{Lemma}\label{limit gamma'}
Let $R\subset K[X]$ be an integral domain with quotient field $K(X)$. Suppose $R\subset V_{\alpha,\gamma'}$ for some $\alpha\in V$ and $\gamma'\in\Gamma_v$ and let $\gamma=\inf\{\gamma'\in \Gamma_v \mid R\subset V_{\alpha,\gamma'}\}\in\R$. Then $R\subset V_{\alpha,\gamma}$.
\end{Lemma}
In particular, $\gamma$ is a minimum if and only if $\gamma\in\Gamma_v$. Note that, if $V$ is nondiscrete, it may well be that $\gamma\in\R\setminus\Gamma_v$. 
\begin{proof}
Let $f\in R$, with $f(X)=a_0+a_1(X-\alpha)+\ldots+a_d(X-\alpha)^d$. Then 
$$f\in V_{\alpha,\gamma'}\Leftrightarrow \inf\{v(a_i)+i\gamma' \mid i=0,\ldots,d\}\geq0\Leftrightarrow a_0\in V,\gamma'\geq-\frac{v(a_i)}{i},i=1,\ldots,d$$
Since $\gamma$ is the infimum of the $\gamma'$ with the above property in particular we have 
$$a_0\in V,\gamma\geq-\frac{v(a_i)}{i},i=1,\ldots,d$$
that is, $v_{\alpha,\gamma}(f)=\inf\{v(a_i)+i\gamma\mid i=0,\ldots,d\}\geq0\Leftrightarrow f\in V_{\alpha,\gamma}$.
\end{proof}
By Lemma \ref{polynomial ring and int balls}, the next theorem shows that if  $B(\alpha,\gamma)$ is the largest ball centered in $\alpha$ contained in the polynomial closure $\overline{S}$ of $S$, then there exists a pseudo-monotone sequence in $S$ with breadth $\gamma$ and pseudo-limit in $V$, which is equal either to $\alpha$ or to $\alpha+t$, where $t\in V$ has valuation $\gamma$.  This result is the desired converse to Proposition \ref{generalized pseudo sequence and polynomial closure}.

\begin{Thm}\label{existence of generalized pseudo-sequence}
Let $S\subseteq V$ be a subset such that $\Int(S,V)\subset V_{\alpha,\gamma'}$, for some $\alpha\in V$ and $\gamma'\in\Gamma_v$. Let $\gamma=\inf\{\gamma'\in\Gamma_v \mid \Int(S,V)\subset V_{\alpha,\gamma'}\}\in\R$. Then there exists a pseudo-monotone sequence $E=\{s_n\}_{n\in\N}\subseteq S$ with breadth $\gamma$ such that one of the following conditions holds:
\begin{itemize}
\item[1)] $E\subseteq S_{\alpha,<\gamma}$ is pseudo-convergent with pseudo-limit $\alpha$.
\item[2)] $E\subseteq S_{\alpha,>\gamma}$ is pseudo-divergent with pseudo-limit $\alpha$. 
\item[3)] $E\subseteq S_{\alpha,\gamma}$ is pseudo-divergent with a pseudo-limit $s\in S_{\alpha,\gamma}$. 
\item[4)] $E\subseteq S_{\alpha,\gamma}$ is pseudo-stationary with pseudo-limit $\alpha$.
\end{itemize}
Moreover, condition 1) holds if and only if $\sup\{v(s-\alpha) \mid s\in S_{\alpha,<\gamma}\}=\gamma$, condition 2) holds if and only if $\inf\{v(s-\alpha) \mid s\in S_{\alpha,>\gamma}\}=\gamma$ and conditions 3) or 4) hold if and only if $B(\alpha,\gamma)\subseteq\overline{S_{\alpha,\gamma}}$. In these last two cases, $\gamma$ is a minimum$\Leftrightarrow\gamma\in\Gamma_v$.

In particular, if $V$ is discrete, $\gamma$ is a minimum and only case 4) holds.
\end{Thm}
\begin{Rem}
Note that by Lemma \ref{limit gamma'} we have $\Int(S,V)\subset V_{\alpha,\gamma}$ either in the case $\gamma\in\Gamma_v\Leftrightarrow \gamma$ is a minimum or $\gamma\in\R\setminus\Gamma_v$. Note also that in case 3), where $s\in S_{\alpha,\gamma}$ is a pseudo-limit  of a pseudo-divergent sequence $E=\{s_n\}_{n\in\N}\subseteq S_{\alpha,\gamma}$, we have $V_{\alpha,\gamma}=V_{s,\gamma}$, by Remark \ref{equality Valphagamma}.
\end{Rem}
\begin{proof}
Since by Theorem \ref{criterion 1 general case} $\Int(S,V)$ is not a Pr\"ufer domain, by 
Theorem  \ref{IntVPrufer} either $V$ is not discrete or its residue field is infinite, otherwise $\Int(V)$ would be Pr\"ufer and in particular its overring $\Int(S,V)$ would be Pr\"ufer, too.

We consider the following real numbers: 
\begin{align*}
\gamma_1=&\sup\{v(s-\alpha) \mid s\in S_{\alpha,<\gamma}\}\\
\gamma_2=&\inf\{v(s-\alpha) \mid s\in S_{\alpha,>\gamma}\}
\end{align*}
Clearly, we have $\gamma_1\leq\gamma\leq\gamma_2$. Since $S_{\alpha,>\gamma}\subseteq B(\alpha,\gamma_2)$ and every closed ball is polynomially closed (Theorem \ref{Thm Chabert}), we have:
\begin{equation}\label{Salpha>gamma}
\overline{S_{\alpha,>\gamma}}\subseteq B(\alpha,\gamma_2)
\end{equation}
If $\gamma_1=\gamma$, then there exists a pseudo-convergent sequence $\{s_n\}_{n\in\N}\subseteq S_{\alpha,<\gamma}$ with pseudo-limit $\alpha$ and breadth $\gamma$,  that is:
$$v(s_n-\alpha)<v(s_{n+1}-\alpha)\nearrow\gamma$$
Similarly, if $\gamma_2=\gamma$, then there exists a pseudo-divergent sequence $\{s_n\}_{n\in\N}\subseteq S_{\alpha,>\gamma}$ with $\alpha$ as pseudo-limit and breadth $\gamma$:
$$v(s_n-\alpha)>v(s_{n+1}-\alpha)\searrow\gamma$$
Hence, if either $\gamma_1=\gamma$ or $\gamma_2=\gamma$ we are done.

Suppose from now on that
$$\gamma_1<\gamma<\gamma_2.$$
We are going to show that under these conditions $\gamma$ is a minimum (or, equivalently, $\gamma\in\Gamma_v$), by means of the fact that $S_{\alpha,\gamma}$ is non-empty. We claim first that $\{v(s-\alpha)\mid s\in S_{\alpha,<\gamma})\}$ is finite; in fact, if that were not true, there would exist a sequence $\{s_n\}_{n\in\N}\subseteq S_{\alpha,<\gamma}$ either pseudo-convergent or pseudo-divergent  with breadth $\gamma'<\gamma$, so that $\Int(S,V)\subset V_{\alpha,\gamma'}$ by Proposition \ref{generalized pseudo sequence and polynomial closure}, contrary to the assumption on $\gamma$. Therefore $\gamma_1$ is a maximum and we may assume that:
$$\{v(s-\alpha)\mid s\in S_{\alpha,<\gamma}\}=\{\gamma_1,\ldots,\gamma_r\},\;\;\gamma_r<\ldots<\gamma_1<\gamma.$$
For each $i=1,\ldots,r$ we set:
$$S_{\alpha,\gamma_i}=\{s\in S_{\alpha,<\gamma} \mid v(s-\alpha)=\gamma_i\}$$
so that
$$S_{\alpha,<\gamma}=\bigcup_{i=1,\ldots,r}S_{\alpha,\gamma_i}.$$
For each $i=1,\ldots,r$, there is no pseudo-stationary sequence $\{s_n\}_{n\in\N}\subset S_{\alpha,\gamma_i}$ with breadth $\gamma_i$ and pseudo-limit $\alpha$, otherwise by Proposition \ref{generalized pseudo sequence and polynomial closure} $\Int(S,V)\subset V_{\alpha,\gamma_i}$, in contradiction with the definition of $\gamma$. Hence, for each $i\in\{1,\ldots,r\}$, there exist finitely many $s_{i,j}\in S_{\alpha,\gamma_i}$, $j\in I_i$,  such that the following holds:
$$\forall s\in S_{\alpha,\gamma_i}, \exists j\in I_i \textnormal{ such that }v(s-s_{i,j})>\gamma_i.$$ 
For each $j\in I_i$, we set $S_{\alpha,\gamma_i,j}=\{s\in S_{\alpha,\gamma_i} \mid v(s-s_{i,j})>\gamma_i\}$ and $\gamma_{i,j}=\inf\{v(s-s_{i,j}) \mid s\in S_{\alpha,\gamma_i,j}\}$. If $\gamma_{i,j}=\gamma_i$, then there exists a pseudo-divergent sequence with  pseudo-limit $s_{i,j}$ and breadth $\gamma_i$ so that by Proposition \ref{generalized pseudo sequence and polynomial closure} we would have $B(s_{i,j},\gamma_i)=B(\alpha,\gamma_i)\subseteq\overline{S}$ (recall that $v(\alpha-s_{i,j})=\gamma_i$), which is equivalent to $\Int(S,V)\subset V_{\alpha,\gamma_i}$ by Lemma \ref{polynomial ring and int balls}, contrary to the assumption on $\gamma$. Thus, $\gamma_{i,j}>\gamma_i$ for all $j\in I_i$ (and for all $i\in 1,\ldots,r$). Finally, we have showed that
$$S_{\alpha,\gamma_i}\subseteq\bigcup_{j\in I_i}B(s_{i,j},\gamma_{i,j})$$
so that
\begin{equation}\label{Salpha<gamma}
S_{\alpha,<\gamma}=\bigcup_{i=1,\ldots,r}S_{\alpha,\gamma_i}\subseteq\bigcup_{i=1,\ldots,r}\bigcup_{j\in I_i}B(s_{i,j},\gamma_{i,j})
\end{equation}
Let $\gamma'\in\Gamma_v$ be such that $\gamma\leq\gamma'<\gamma_2$ and $\Int(S,V)\subset V_{\alpha,\gamma'}$; note that if $\gamma$ is not a min, then in particular $\Gamma_v$ is non-discrete and we may choose $\gamma'\in\Gamma_v$ such that $\gamma<\gamma'<\gamma_2$; if $\gamma$ is a min, then we choose $\gamma'=\gamma$. 
Since the polynomial closure is a topological closure by Theorem \ref{Thm Chabert}, by Lemma \ref{polynomial ring and int balls} we have
$$B(\alpha,\gamma')\subseteq \overline{S}=\overline{S_{\alpha,<\gamma}}\cup \overline{S_{\alpha,\gamma}}\cup\overline{S_{\alpha,>\gamma}}$$
We claim that $\partial B(\alpha,\gamma')\subseteq \overline{S_{\alpha,\gamma}}$. In fact, let $\beta\in V$ be such that $v(\beta-\alpha)=\gamma'$. If $\beta\in \overline{S_{\alpha,<\gamma}}$ then by (\ref{Salpha<gamma}) we have $v(\beta-s_{i,j})\geq\gamma_{i,j}$ for some $i\in\{1,\ldots,r\}$ and $j\in I_i$, which implies that $v(\alpha-s_{i,j})=v(\alpha-\beta+\beta-s_{i,j})>\gamma_i$, a contradiction. Similarly, $\beta$ is not in $\overline{S_{\alpha,>\gamma}}$, by(\ref{Salpha>gamma}) and the fact that $\gamma'<\gamma_2$. Therefore $\partial B(\alpha,\gamma')\subseteq \overline{S_{\alpha,\gamma}}$, as claimed. In particular, $S_{\alpha,\gamma}\not=\emptyset$, $\gamma\in\Gamma_v$ and so $\gamma=\min\{\gamma'\in\Gamma_v \mid \Int(S,V)\subset V_{\alpha,\gamma'}\}$. We may therefore assume that $\gamma'=\gamma$. By Lemma \ref{pol closure frontier} (recall that either $V$ is not discrete or its residue field is infinite) we have $ B(\alpha,\gamma)\subseteq \overline{S_{\alpha,\gamma}}$, so that,  by Proposition \ref{polynomial closure Salpha=gamma}, there exists a pseudo-monotone sequence $E=\{s_n\}_{n\in\N}\subseteq S_{\alpha,\gamma}$ with breadth $\gamma$  such that either $E$ is pseudo-stationary and has pseudo-limit $\alpha$ or $E$ is  pseudo-divergent and has  pseudo-limit in $S_{\alpha,\gamma}$. The proof is now complete.
\end{proof}

As we have already said, if $\Int(S,V)$ is not Pr\"ufer  then there might be residually transcendental valuation overrings which are not of the form $V_{\alpha,\gamma}$ with $(\alpha,\gamma)\in K\times\Gamma_v$. For example, if $E=\{s_n\}_{n\in\N}\subset V$ is a pseudo-convergent sequence of algebraic type without pseudo-limits in $K$, then $\Int(E,V)$ is not Pr\"ufer by Theorem \ref{ThmLW}; by Theorem \ref{existence of generalized pseudo-sequence} is not difficult to show that $\Int(S,V)\not\subset V_{\alpha,\gamma}$ for every $(\alpha,\gamma)\in K\times\Gamma_v$. However, we may reduce our discussion to this case by means of  Proposition \ref{IntSVIntSW}, as the following proposition shows.

\begin{Prop}\label{reduction to K}
Let $S\subseteq V$ be such that $\Int(S,V)\subset V_{\alpha,\gamma'}=W_{\alpha,\gamma'}\cap K(X)$ for some $(\alpha,\gamma')\in F\times\Gamma_{w}$, where $F$ is a finite extension of $K$ and $W$ is a valuation domain of $F$ lying over $V$. Let $\gamma=\inf\{\gamma'\in\Gamma_w \mid \Int(S,V)\subset V_{\alpha,\gamma'}=W_{\alpha,\gamma'}\cap K(X)\}$. Then there exists a pseudo-monotone sequence $E=\{s_n\}_{n\in\N}\subseteq S$ with breadth $\gamma$ and pseudo-limit which is equal either to  $\alpha$ or belongs to $\{x\in W \mid w(x-\alpha)=\gamma\}$.

If $V$ is a DVR, then $E$ is pseudo-stationary, the breadth $\gamma$ is in $\Gamma_v$ and there exists $\beta\in V$ which is a pseudo-limit of $E$, so that, in particular, $V_{\alpha,\gamma}=V_{\beta,\gamma}$.
\end{Prop}
\begin{proof}
Suppose that $\Int(S,V)\subset V_{\alpha,\gamma'}=W_{\alpha,\gamma'}\cap K(X)$ as in the assumptions of the proposition. By Proposition \ref{IntSVIntSW}, $\Int(S,W)\subset W_{\alpha,\gamma'}$. Note that, by Lemma \ref{overrings of V[X]}, $\alpha\in W$ and $\gamma'\geq 0$. By Theorem \ref{existence of generalized pseudo-sequence}, there exists a pseudo-monotone sequence $\{s_n\}_{n\in\N}\subseteq S$ whose breadth is $\gamma$ and has pseudo-limit which is either $\alpha$ or belongs to $\{x\in W \mid w(x-\alpha)=\gamma\}$.

In the case $V$ is discrete, by Theorem \ref{existence of generalized pseudo-sequence} $E$ is necessarily pseudo-stationary with breadth $\gamma$ and $\alpha$ is a pseudo-limit of $E$. Since the $s_n$'s are elements of $K$, $w(s_n-s_m)=v(s_n-s_m)$, so that $\gamma\in\Gamma_v$. Moreover, by \S \ref{Remark pseudo-stationary}, any  element of $E=\{s_n\}_{n\in\N}$ can be considered as a pseudo-limit of $E$, so the last equality follows from Remark \ref{equality Valphagamma}.
\end{proof}
Finally, the next theorem characterizes the subsets $S$ of $V$ for which $\Int(S,V)$ is Pr\"ufer. We give first the following generalization of the definition of pseudo-limit. 

\begin{Def}
Let $E=\{s_n\}_{n\in\N}$ be a pseudo-monotone sequence of $K$ and $\alpha\in\overline{K}$. We say that $\alpha$ is a pseudo-limit of $E$ if there exists a valuation $w$ of $K(\alpha)$ which lies above $v$ such that $\alpha$ is a pseudo-limit of $E$ with respect to $w$ (clearly, $E$ is a pseudo-monotone sequence with respect to $w$). 
\end{Def}

We recall that for our convention (see \S\ref{pcv}) a pseudo-convergent sequence $E$ has non-zero breadth ideal. With this terminology, the next theorem shows that $\Int(S,V)$ is Pr\"ufer if and only if $S$ does not admit any pseudo-limit in $\overline{K}$. This theorem generalizes the main result of Loper and Werner (Theorem \ref{ThmLW}).
\begin{Thm}\label{final theorem}
Let $S\subseteq V$. Then $\Int(S,V)$ is a Pr\"ufer domain if and only if $S$ does not contain a pseudo-monotone sequence $E=\{s_n\}_{n\in\N}$ which has a  pseudo-limit $\alpha\in\overline{K}$. 
\end{Thm}
\begin{proof}
Suppose there exists a pseudo-monotone sequence $E=\{s_n\}_{n\in\N}\subseteq S$ with breadth $\gamma\in \R$ and pseudo-limit $\alpha\in\overline{K}$. If $E$ is pseudo-stationary, then we know that $\gamma\in\Gamma_v$ and any element $s$ of $E$ is a pseudo-limit of $E$ (\S \ref{Remark pseudo-stationary}). Then by Proposition \ref{generalized pseudo sequence and polynomial closure}, $\Int(S,V)\subset V_{s,\gamma}=V_{\alpha,\gamma}$, so by Theorem \ref{criterion 1 general case}  $\Int(S,V)$ is not Pr\"ufer. Suppose now that $E$ is either pseudo-convergent or pseudo-divergent, and let $F$ be a finite extension of $K$ which contains $\alpha$. Let $W$ be a valuation domain of $F$ lying over $V$ (which is necessarily of rank one) for which $\alpha$ is a pseudo-limit of $E$ (which clearly is a pseudo-monotone sequence with respect to the associated valuation $w$). Clearly,  $\alpha\in W$. By Proposition \ref{generalized pseudo sequence and polynomial closure}, it follows that $\Int(S,W)\subset W_{\alpha,\gamma}$ and contracting down to $K[X]$ we get $\Int(S,V)\subset V_{\alpha,\gamma}=W_{\alpha,\gamma}\cap K(X)$, so, by Theorem \ref{criterion 1 general case} and Remark \ref{breadth not in Gamma_v}, $\Int(S,V)$ is not Pr\"ufer.

Conversely, suppose that $\Int(S,V)$ is not Pr\"ufer. By Theorem \ref{criterion 1 general case},  there exists $(\alpha,\gamma')\in \overline{K}\times\Gamma_{\overline{v}}$ and an extension $\overline{W}$ of $V$ to $\overline K$ such that $\Int(S,V)\subset V_{\alpha,\gamma'}^{\overline{W}}$. As in Remark \ref{example}, let $F$ be a finite extension of $K$ and $W=\overline{W}\cap F$  such that $(\alpha,\gamma')\in F\times\Gamma_w$. Let $\gamma=\inf\{\gamma'\in\Gamma_w \mid \Int(S,V)\subset V_{\alpha,\gamma'}=W_{\alpha,\gamma'}\cap K(X)\}\in \R$. By Proposition \ref{reduction to K}, there exists a pseudo-monotone sequence $E=\{s_n\}_{n\in\N}\subseteq S$ with breadth $\gamma$ and pseudo-limit which is equal either to $\alpha$ or to $\alpha+t$, where $t\in W$ is such that $w(t)=\gamma$.
\end{proof}

We summarize here some known results and new characterizations of when $\Int(S,V)$ is a Pr\"ufer domain, when $V$ is a DVR. Recall that, as already remarked by Loper and Werner, if $V$ is a non-discrete rank one valuation domain,  there are subsets $S$ of $V$ which are not precompact but $\Int(S,V)$ is Pr\"ufer (see the Introduction).
\begin{Cor}
Let $V$ be a DVR and $S\subseteq V$. Then the following conditions are equivalent:
\begin{itemize}
\item[i)] $\Int(S,V)$ is Pr\"ufer.
\item[ii)] there is no pseudo-stationary sequence contained in $S$.
\item[iii)] $S$ is precompact.
\item[iv)] there is no $(\alpha,\gamma)\in V\times\Gamma_v$ such that $\Int(S,V)\subset V_{\alpha,\gamma}$.
\end{itemize}
\end{Cor}
\begin{proof}
It is easy to see that ii) and iii) are equivalent if $V$ is discrete. In fact, $S$ is precompact if and only if $S$ modulo $M^n$ is finite for each $n\geq 1$ (\cite[Proposition 1.2]{CCL}). Now, if the latter condition holds, then there cannot be any pseudo-stationary sequence in $V$ by \S \ref{Remark pseudo-stationary}. Conversely, if $S$ modulo $M^n$ is infinite for some $n\geq 1$, then there exists $\{s_m\}_{m\in\N}\subseteq S$ such that $v(s_m-s_k)<n$ for each $m\not=k$. Since the values $\{v(s_m-s_k) \mid m\not=k\}$ are finite, we can extract a subsequence from $\{s_m\}_{m\in\N}$ which is pseudo-stationary, as claimed.

The fact that i) and iv) are equivalent follows from Proposition \ref {reduction to K} and Theorem \ref{existence of generalized pseudo-sequence}. 

We show then that i) and ii) are equivalent. By Proposition \ref{generalized pseudo sequence and polynomial closure}, if there is a pseudo-stationary sequence with breadth $\gamma$ contained in $S$, then $\Int(S,V)\subset V_{\alpha,\gamma}$, so by Lemma \ref{ValphagammacapK[X] not Prufer} $\Int(S,V)$ is not Pr\"ufer. Suppose now that $\Int(S,V)$ is not Pr\"ufer: by Theorem \ref{criterion 1 general case}, Remark \ref{example} and Lemma \ref{overrings of V[X]} there exist a finite extension $F$ of $K$, a valuation domain $W$ of $F$ extending $V$ and $(\alpha,\gamma)\in W\times\Gamma_{w}$ such that $\Int(S,V)\subset V_{\alpha,\gamma}=W_{\alpha,\gamma}\cap K(X)$. Suppose also that $\gamma\in\Gamma_w$ is minimal with this property. Then by Proposition \ref{reduction to K} it would follow that $\gamma\in \Gamma_v$ and there exists a pseudo-stationary sequence in $S$ with breadth $\gamma$, a contradiction.
\end{proof}

Actually, iii) implies ii) also when $V$ is not discrete. More generally, when $S$ is a precompact subset of $V$, then there is no pseudo-monotone sequence contained in $S$, so that by Theorem \ref{final theorem} we get again the result of Theorem \ref{ThmCCL}.

\subsection*{\textbf{Acknowledgements}}

This research has been supported by the grant "Ing. G. Schirillo" of the Istituto Nazionale di Alta Matematica and also by the University of Padova.

The author wishes to thank the anonymous referee for improving the quality of the paper.

The author wishes to thank also Jean-Luc Chabert for pointing out some inaccuracies.

\end{document}